\documentclass[12pt]{amsart}
\usepackage{graphicx}
\usepackage{amssymb}
\usepackage{amsmath}
\usepackage{amscd}
\usepackage{upgreek}
\usepackage{mathrsfs}
\usepackage{stmaryrd}
\usepackage{longtable}
\usepackage{txfonts}
\usepackage[T1]{fontenc}
\usepackage{eucal}
\usepackage{titletoc}
\usepackage{wrapfig}
\usepackage{float}
\usepackage{xypic}
\usepackage{dsfont}
\usepackage{xcolor}
\usepackage{color}
\usepackage[colorlinks,linkcolor=blue,anchorcolor=blue,
citecolor=blue]{hyperref}
\usepackage{hyperref}
\usepackage[hmarginratio=1:1,top=32mm,columnsep=20pt]{geometry}
	
\allowdisplaybreaks	
\usepackage{mathptmx}
\DeclareMathAlphabet{\mathcal}{OMS}{cmsy}{m}{n}
\DeclareSymbolFont{largesymbols}{OMX}{cmex}{m}{n}

\newtheorem{theorem}{Theorem}[section]
\newtheorem{prop}[theorem]{\rm \textsc{Proposition}}

\newtheorem{coro}[theorem]{\rm \textsc{Corollary}}

\newtheorem{thm}[theorem]{\it \textsc{Theorem}}

\newtheorem{rem}[theorem]{\rm \textsc{Remark}}
\newtheorem{exam}[theorem]{\rm\textsc{Example}}

\newcommand{\N}{\mathbb{N}} % the set of nonnegative integers

\newcommand{\F}{\mathbb{F}}

\newcommand{\ra}{\longrightarrow}

\newcommand{\GL}{{\rm GL}} % the general linear groups
\newcommand{\SL}{{\rm SL}}
\newcommand{\Sp}{{\rm Sp}}
\begin{document}
\setlength{\oddsidemargin}{0cm}
\setlength{\evensidemargin}{0cm}

%%%%%%%%%%%%%%%%%%%%%%%%%%%%%%%%%%% TITLE, AUTHOR AND ADDRESS
\title{\scshape On modular vector invariant fields}
\author{\scshape Ying Han}
\address{School of Mathematics and Statistics, Northeast Normal University, Changchun 130024, P.R. China}
\email{hany444@nenu.edu.cn}
\author{\scshape Runxuan Zhang}
\address{School of Mathematics and Statistics, Northeast Normal University, Changchun 130024, P.R. China.}
\email{zhangrx728@nenu.edu.cn}

%\date{\today}
\def\shorttitle{On modular vector invariant fields}

%%%%%%%%%%%%%%%%%%%%%%%%%%%%%%%%%%%ABSTRACT, MSC AND KEYWORDS
\begin{abstract}
Let $\F_q$ be a finite field of any characteristic and $\GL(n,\F_q)$ the general linear group over $\F_q$. Suppose $W$ denotes the standard representation of $\GL(n,\F_q)$ and $\GL(n,\F_q)$ acts diagonally on
the direct sum of $W$ and its dual space $W^*$. Let $G$ be any subgroup of $\GL(n,\F_q)$. Suppose the invariant field $\F_q(W)^G=\F_q(f_1,f_2,\dots,f_k)$, where $f_1,f_2,\dots,f_k\in\F_q[W]^G$ are homogeneous invariant polynomials.
We prove that there exist homogeneous polynomials $\ell_1,\ell_2,\dots,\ell_{n}$ in the invariant ring $\F_q[W\oplus W^*]^{G}$ such that the invariant field $\F_q(W\oplus W^*)^{G}$ is  generated by $\{f_1,f_2,\dots,f_k,\ell_1,\ell_2,\dots,\ell_{n}\}$ over $\F_q$.
\end{abstract}

\subjclass[2010]{13A50. }
\keywords{Modular invariants; vector invariant fields; finite classical groups. %\\ ~~~$^{*}$Corresponding author. 
}

\maketitle
\baselineskip=18pt

%%%%%%%%%%%%%%%%%%%%%%%%%%%%%%%%%%%%%CONTENTS
%\textcolor{blue}{\tableofcontents{}}

\dottedcontents{section}[1.16cm]{}{1.8em}{5pt}
\dottedcontents{subsection}[2.00cm]{}{2.7em}{5pt}
\dottedcontents{subsubsection}[2.86cm]{}{3.4em}{5pt}

\section{\scshape Introduction}

\setcounter{equation}{0}
\renewcommand{\theequation}
{1.\arabic{equation}}
\setcounter{theorem}{0}
\renewcommand{\thetheorem}
{1.\arabic{theorem}}

Let $\F$ be any field and $V$  a finite-dimensional vector space over $\F$.
Any finite subgroup $G\leq \GL(V)$ acts linearly on $V$ and the dual space $V^*$. We extend the action of $G$ on $V^*$ multiplicatively to the symmetric algebra $\F[V]$ on $V^*$ and the field of fractions $\F(V)$ of $\F[V]$.
The invariant ring $\F[V]^G$ is the subring of $\F[V]$ consisting of all elements fixed by every $\sigma\in G$, which together with the fixed field $\F(V)^G$ are two main objects of the study in the invariant theory. It is well-known that $\F(V)^G$ is equal to the field of fractions of $\F[V]^G$; see \cite{Ben1993,CW2011,DK2015} for general references of invariant theory of finite groups.

Suppose $\dim(V)=m$ and $\F(V)^G$ is purely transcendental over $\F$, i.e., the Noether problem has a positive answer for $\F(V)^G$.  A natural question involved in $\F[V]^G$ and $\F(V)^G$ asks
\begin{equation}\tag{$\dag$}
\label{ques}
\textit{whether there exist } \ell_1,\ell_2,\dots,\ell_m\in \F[V]^G\textit{ such that }\F(V)^G=\F(\ell_1,\ell_2,\dots,\ell_m)?
\end{equation}
This question has been studied extensively; see for example \cite{CC2007,Kan2006,FF2016} for $p$-groups and \cite{Chu1997,TW2006,CT2018,CW2017,CW2018}
for finite classical groups.

Here we are interested  in the modular case for the question. We assume that $\F=\F_q$ is a finite field of any characteristic and $W$  an $n$-dimensional vector space over $\F_q$.
Let $G\leq\GL(W)$ be any subgroup.  We consider the diagonal action of $G$ on the direct sum  $V=W\oplus W^*$.
The purpose of this   note is to prove the following result, which answers the above-mentioned question (\ref{ques})
for some particular cases and generalizes a result due to Chen and Wehlau \cite[Proposition 5]{CW2017}.

\begin{thm}\label{mt} Let $W$ be an $n$-dimensional vector space over a finite field $\F_q$ and
 $G$  any subgroup of $\GL(W)$. Suppose $\F_q(W)^G=\F_q(f_1,f_2,\dots,f_k)$, where $f_1,f_2,\dots,f_k\in\F_q[W]^G$ are homogeneous.
Then there exists a set of homogeneous invariant polynomials $\{\ell_1,\ell_2,\dots,\ell_{n}\}\subseteq \F_q[W\oplus W^*]^{G}$ such that $\F_q(W\oplus W^*)^{G}$ is  generated by $\{f_1,f_2,\dots,f_k,\ell_1,\ell_2,\dots,\ell_{n}\}$ over $\F_q$.
\end{thm}

Section \ref{sect2} contains some basic constructions and the proof of Theorem \ref{mt}.
In Section \ref{sect3}, we present two examples on finite classical groups.

\section{\scshape Proof of Theorem \ref{mt}} \label{sect2}

\setcounter{equation}{0}
\renewcommand{\theequation}
{2.\arabic{equation}}
\setcounter{theorem}{0}
\renewcommand{\thetheorem}
{2.\arabic{theorem}}

We first recall some constructions on modular invariants of one vector and one covectors appeared in \cite{BK2011} and \cite{CW2017}.  Let $\{y_{1},y_{2},\dots,y_{n}\}$ be a basis of $W$ and
$\{x_{1},x_{2},\dots,x_{n}\} $ a basis of $W^*$ dual to $\{y_{1},y_{2},\dots,y_{n}\}$. We may assume that
$\F_{q}[W]=\F_{q}[x_{1},x_{2},\dots,x_{n}]$, $\F_{q}[W^*]=\F_{q}[y_{1},y_{2},\dots,y_{n}]$, and
$\F_{q}[W\oplus W^{\ast}]=\F_{q}[x_{1},x_{2},\dots,x_{n},y_{1},y_{2}, \dots,y_{n}].$
We define an involution of $\F_q$-algebras $\ast :\F_{q}[W\oplus W^{\ast}]\ra\F_{q}[W\oplus W^{\ast}], f\mapsto f^*$ by
$$x_{i}\longmapsto y_{n+1-i}\textrm{ and }y_{i}\longmapsto x_{n+1-i}$$
and we define two homomorphisms of $\F_q$-algebras:
\begin{eqnarray*}
 F:\F_{q}[W\oplus W^{\ast}]\ra\F_{q}[W\oplus W^{\ast}]&\textrm{by}&  x_{i}\longmapsto x_{i}^{q},  y_{i} \longmapsto y_{i};\\
 F^*:\F_{q}[W\oplus W^{\ast}]\ra\F_{q}[W\oplus W^{\ast}]&\textrm{by}&  x_{i}\longmapsto x_{i},  y_{i} \longmapsto y_{i}^q
\end{eqnarray*}
for $i=1,2,\dots,n$. There exist some natural $\GL(W)$-invariants:
\begin{eqnarray*}
u_{0}&:=&\sum_{j=1}^{n} x_{j}y_{j} \\
u_{i}&:=& F^{i}(u_{0})=\sum_{j=1}^{n} x_{j}^{q^{i}}y_{j} \\
u_{-i}&:=&(F^{\ast})^{i}(u_{0})=\sum_{j=1}^{n}x_{j}y_{j}^{q^{i}}
\end{eqnarray*}
for all $ i \in \mathbb{Z}$. Note that $u_{-i}=u_{i}^{\ast}$ for all $i.$

\begin{proof}[Proof of Theorem \ref{mt}]

By \cite[Proposition 5]{CW2017}, we see that $\F_{q}(W \oplus W^{\ast})^{\GL(W)}$ is minimally generated by $\{c_{n0},c_{n1},\dots,c_{n,n-1},u_{0},u_{1},\dots,u_{n-1}\}$, where $c_{n0},c_{n1},\dots,c_{n,n-1}$ denote the Dickson invariants; see for example \cite[Section 2]{CW2017}. Note that $\F_q(c_{n0},c_{n1},\dots,c_{n,n-1})=\F_q(W)^{\GL(W)}\subseteq\F_q(W)^G=\F_q(f_1,f_2,\dots,f_k)$.
Thus every $c_{ni}(0\leq i\leq n-1)$ can be expressed rationally by $f_{1},f_{2},\dots,f_{k}$ over $\F_q$, it follows that
$$\F_{q}(W \oplus W^{\ast})^{\GL(W)}\subseteq \F_q(f_{1},f_{2},\dots,f_{k},u_{0},u_{1},\dots,u_{n-1})=:E.$$
Since $u_{0},u_{1},\dots,u_{n-1}$ are $\GL(W)$-invariant, they also belong to $\F_{q}(W \oplus W^{\ast})^{G}$. This implies
$E\subseteq \F_{q}(W \oplus W^{\ast})^{G}$, and thus we derive a tower of field extensions:
$$\F_{q}(W \oplus W^{\ast})^{\GL(W)}\subseteq E\subseteq \F_{q}(W \oplus W^{\ast})^{G}.$$
Let $N$ be the Galois group of $\F_{q}(W \oplus W^{\ast})$ over $E$. Then we have the following situation:
$$
\xymatrix{
\F_q(W\oplus W^*)^{\GL(W)} ~\ar@{--}[d] \ar@{^{(}->}[r] &~E~ \ar@{--}[d] \ar@{^{(}->}[r]& ~\F_q(W\oplus W^*)^{G}~ \ar@{--}[d] \ar@{^{(}->}[r]&  ~\F_q(W\oplus W^*) \ar@{--}[d]\\
\GL(W)~& ~N~\ar@{_{(}->}[l]  &~G~ \ar@{_{(}->}[l] & ~1 \ar@{_{(}->}[l]
}
$$
We \textit{claim} that $E=\F_q(W\oplus W^*)^{G}$. By Galois theory, $E=\F_q(W\oplus W^*)^{G}$ if and only if $G=N$. As $G\subseteq N$, it suffices to show that $N\subseteq G$.
Recall that $G$ is the Galois group of $\F_{q}(W)$ over $\F_{q}(W)^{G}=\F_q(f_1,f_2,\dots,f_k)$.
For any $\sigma \in N\leq \GL(W)$, since it fixes every element of $E$, it also
fixes every element of $\F_{q}(W)^{G}$. Thus $\sigma\in G$ and the claim holds.
This implies that $\F_q(W\oplus W^*)^{G}$ is  generated by $\{f_{1},f_{2},\dots,f_{k},u_{0},u_{1},\dots,u_{n-1}\}$ over $\F_q$.
\end{proof}

\section{\scshape Examples} \label{sect3}

\setcounter{equation}{0}
\renewcommand{\theequation}
{3.\arabic{equation}}
\setcounter{theorem}{0}
\renewcommand{\thetheorem}
{3.\arabic{theorem}}

In this section we present two examples on finite classical groups.

\begin{exam}{\rm
It is well-known that $\F_q(W)^{\SL(W)}=\F_q(d_{nn},c_{n1},\dots,c_{n,n-1})$ is purely transcendental over $\F_q$, where $d_{nn}^{q-1}=c_{n0}$. By Theorem \ref{mt}, we see that
$$\F_q(W\oplus W^*)^{\SL(W)}=\F_q(d_{nn},c_{n1},\dots,c_{n,n-1},u_{0},u_{1},\dots,u_{n-1}).$$
This result also has been obtained by \cite[Proposition 2.2]{CW2018}.
}\end{exam}

\begin{exam}{\rm
Suppose $L=[\begin{smallmatrix}
               0 & 1 \\
               -1 & 0 \\
             \end{smallmatrix}]
$ and $H=\textrm{diag}\{L,L,\dots,L\}$ is a $2n\times 2n$ alternate matrix over $\F_q$.
A matrix $T\in \GL(2n,\F_q)$ is said to be \textit{symplectic} with respect to $H$, if $T'HT=H$, where
$T'$ denotes the transpose of $T$.
The symplectic group $\Sp(2n,\F_q)$ is a subgroup of $\GL(2n,\F_q)$ consisting of all symplectic matrices.
Let $W$  denote the standard representation of $\GL(2n,\F_q)$. Choose a basis $\{x_1,x_2,\dots,x_{2n}\}$ for $W^*$ and make the convention that $X^k:=(x_1^{q^k},x_2^{q^k},\dots,x_{2n}^{q^k})$ for $k\in \N$. We define
$$P_k:=X^0\cdot H\cdot (X^k)',$$
where $(X^k)'$ denotes the transpose of the vector $X^k$.
By \cite[Corollary 3]{Chu1997}, we see that  $\F_q(W)^{\Sp(2n,\F_q)}$ is purely transcendental and minimally generated by $\{P_1,P_2,\dots,P_{2n}\}$ over $\F_q$. Applying Theorem \ref{mt}, we have
\begin{coro}
$\F_{q}(W \oplus W^{\ast})^{\Sp(2n,\F_q)}=\F_q(P_{1},P_{2},\dots,P_{2n},u_{0},u_{1},\dots,u_{2n-1})$.
\end{coro}

To gain another minimal generating set for $\F_{q}(W \oplus W^{\ast})^{\Sp(2n,\F_q)}$, we need to recall the notion of $\ast$-stable subgroups appeared in \cite{BK2011}. Let $J_{2n}=\left[\begin{smallmatrix}
                     0 & \ldots & 1 \\
                     \vdots & \ddots  & \vdots \\
                     1 & \ldots & 0 \\
                   \end{smallmatrix}\right]$ be a $2n\times 2n$-matrix over $\F_q$. Then
$\sigma\longmapsto \sigma^{\ast}:=J_{2n}\cdot (\sigma^{-1})'\cdot J_{2n}$ gives rise to an automorphism of $\GL(W)$.
 Note that $(\sigma\cdot f)^{\ast}=\sigma^{\ast} \cdot f^{\ast}$ for all $\sigma\in\GL(W)$
 and $f\in\F_q[W\oplus W^*]$. We say that a subgroup $G\leq \GL(W)$ is \textit{$\ast$-stable} if $\ast(G)\subseteq G$.

 \begin{prop}\label{prop3.4}
The symplectic group $\Sp(2n,\F_{q})$ is $\ast$-stable.
\end{prop}

\begin{proof}
For any $T\in\Sp(2n,\F_q)$, it suffices to show that $T^*\in \Sp(2n,\F_q)$, i.e., $(T^*)'\cdot H\cdot T^*=H$.
A direct computation shows that $H^{-1}=H'$ and $J_{2n}'\cdot H\cdot J_{2n}=H'$. Note that $J_{2n}'=J_{2n}$.
As $T'\cdot H\cdot T=H$, we have $(T'\cdot H \cdot T)^{-1}=T^{-1}\cdot H^{-1}\cdot (T')^{-1}=H^{-1}$. Thus
$$ T^{-1}=H^{-1}\cdot T'\cdot H=H'\cdot T'\cdot H.$$
Hence, $(T^{\ast})'\cdot H\cdot T^{\ast}=(J_{2n}(T^{-1})'J_{2n})'\cdot H\cdot J_{2n}(T^{-1})'J_{2n}=
J_{2n}T^{-1}J_{2n}\cdot H\cdot J_{2n}(T^{-1})'J_{2n}=J_{2n}T^{-1} \cdot H'\cdot  (T^{-1})'J_{2n}$.
Since $T^{-1} \cdot H'\cdot  (T^{-1})'=(T^{-1}\cdot H\cdot (T^{-1})')'=(H'T'H\cdot H\cdot H'TH)'=
(H'T'H\cdot TH)'=(H'\cdot H \cdot H)'=H'$, it follows that $(T^{\ast})'\cdot H\cdot T^{\ast}=J_{2n}\cdot H'\cdot J_{2n}=H$, as desired.
\end{proof}

From Proposition \ref{prop3.4}, together with the fact that $*$ induces an automorphism of $\F_q[W\oplus W^*]^{\Sp(2n,\F_{q})}$ (see \cite[page 99]{BK2011}), we obtain

\begin{coro}
$\F_{q}(W \oplus W^{\ast})^{\Sp(W)}=\F_q(P_{1}^*,P_{2}^*,\dots,P_{2n}^*,u_{0},u_{-1},\dots,u_{1-2n})$.
\end{coro}
}\end{exam}

\begin{rem}{\rm
Based on the results in \cite{Chu1997} and \cite{TW2006},
we might have similar conclusions for the orthogonal groups, the unitary groups and the pseudo-symplectic groups.
}\end{rem}

\section*{\scshape Acknowledgments}

The authors would like to thank Professor Zhongming Tang and Yin Chen for their helpful comments on an earlier draft of this paper.  This research was partially supported by NNSF of China (No. 11301061). The authors are also grateful to the referee for his/her helpful comments.

\end{document}